\documentclass[smallextended]{article}

\usepackage{amsmath}
\usepackage{amsfonts}
\usepackage{amssymb}
\usepackage{mathrsfs}
\usepackage{amscd}
\usepackage{pb-diagram}
\usepackage{amsthm,mathtools}
\usepackage{color}
\usepackage[all]{xy}
\usepackage{graphicx}
\usepackage{url}
\usepackage{enumerate}
\usepackage[titletoc,title]{appendix}
\usepackage{dsfont}
\usepackage{multirow} 
\numberwithin{equation}{section} 
\usepackage{tikz-cd}
\numberwithin{equation}{section} 
\usepackage{caption}
\DeclareUnicodeCharacter{2212}{-}
\usepackage{tikz}

\usepackage{titlesec}
\titleformat{\subsection}[runin]{\normalsize\bfseries}{\thesubsection}{5pt}{}

\usepackage{tikz}          
\usetikzlibrary{patterns.meta, matrix, arrows, decorations.pathmorphing}
\usetikzlibrary{patterns}

\author{Danuzia Figueir\^edo \\
\small \url{danuzianf@hotmail.com}
\and 
Hale Ayta\c{c} \\
\small \url{aytachale@gmail.com}
\and 
Mathieu Molitor\\
\small \url{pergame.mathieu@gmail.com} \\
\it \small{Instituto de Matem\'{a}tica, Universidade Federal da Bahia}\\
\it \small{Salvador, Brazil}\\ 
}

\title{Exponential families and affine Grassmannians}
\date{}

\begin{document}

\theoremstyle{definition}
\newtheorem{lemma}{Lemma}[section]
\newtheorem{definition}[lemma]{Definition}
\newtheorem{proposition}[lemma]{Proposition}
\newtheorem{corollary}[lemma]{Corollary}
\newtheorem{theorem}[lemma]{Theorem}
\newtheorem{remark}[lemma]{Remark}
\newtheorem{example}[lemma]{Example}
\bibliographystyle{alpha}

\maketitle 


\begin{abstract}
	We establish a one-to-one correspondence between the set of minimal exponential families of dimension $n$ 
	defined on a finite sample space $\Omega$ and the affine Grassmannian associated to an appropriate vector space of functions.
\end{abstract}

\section{Introduction}

	In statistical theory, exponential families are parametric families of distributions with a 
	specific form that considerably simplifies mathematical analysis \cite{brown,casella,nielsen}.
	Their rich properties and inherent tractability make them among the most important statistical models, 
	encompassing many commonly used distributions, including the normal, Poisson, binomial, exponential, 
	gamma, beta, and Bernoulli.

	To formally define exponential families in the discrete case, consider a finite sample space 
	$\Omega=\{x_{0}, ..., x_{m}\}$ equipped with the counting measure. 
	A family of probability distribution functions $\{p_{\theta} : \Omega \to \mathbb{R}\}$ indexed by a parameter 
	$\theta\in \mathbb{R}^{n}$ is an $n$-dimensional exponential family if there exist functions 
	$C, F_{1}, ..., F_{n} : \Omega \to \mathbb{R}$ such that	
	$p_{\theta}(x)=\textup{exp}\{C(x)+\theta_{1}F_{1}(x)+...+\theta_{n}F_{n}(x)-\psi(\theta)\}$ 
	for all $\theta\in \mathbb{R}^{n}$ and $x\in \Omega$. The function $\psi:\mathbb{R}^{n}\to \mathbb{R}$, 
	known as the \textit{log-partition function} or \textit{cumulant generating function}, is uniquely determined by 
	the normalization constraint $\sum_{x\in \Omega}p_{\theta}(x)=1$. 

	The functions $C, F_{1}, ..., F_{n}$ completely determine the exponential family. 
	However, this representation in terms of $C, F_{1}, ..., F_{n}$ is not unique; different 
	functions $C', F_{1}', ..., F_{n}'$ may correspond to the same exponential family. 
	This leads to the problem of identifying the equivalence class of functions that define a 
	given exponential family, a problem not systematically addressed in the literature and the focus 
	of this article.

	We approached this problem by considering the action of a 
	Lie group $G_{n}$ of dimension $(n+1)^{2}$ on the set of all $(n+1)$-tuples $(C,F_{1},...,F_{n})$. 
	We show (Proposition \ref{neknkrnkefnk}) that two such tuples 
	determine the same exponential family if and only if they belong to the same $G_{n}$-orbit. 
	Investigating the corresponding orbit space, we prove our main result (Theorem \ref{nekwnkefnkn}):
	the $G_{n}$-action is free and proper, and the orbit space is diffeomorphic 
	to the affine Grassmannian $\textup{Graff}_{n}(C(\Omega)/\mathbb{R})$ 
	of $n$-dimensional affine subspaces of the vector space $C(\Omega)/\mathbb{R}$. Here 
	$C(\Omega)$ is the space of real-valued functions on $\Omega$ and $\mathbb{R}$ is identified 
	with the space of constant functions on $\Omega$. The proof relies on the fact that $G_{n}$ has the structure of a 
	semidirect product, which allows us to perform a reduction by stages \cite{stages}. 
	As a corollary (Corollary \ref{nfeknkernkfn}), we obtain 
	a one-to-one correspondence between the set of $n$-dimensional 
	(minimal) exponential families on $\Omega$ and the affine Grassmannian $\textup{Graff}_{n}(C(\Omega)/\mathbb{R})$. 
	For instance, if the cardinality of $\Omega$ is $n+1$, then $\textup{Graff}_{n}(C(\Omega)/\mathbb{R})$ consists 
	of a single affine subspace (namely $C(\Omega)/\mathbb{R}$), implying the existence of only one $n$-dimensional (minimal) 
	exponential family on $\Omega$. 

	The paper begins with a review of the essential background, including affine geometry, reduction by stages, 
	and the Stiefel and affine Grassmannian manifolds. Section \ref{neknkrnkefnkw} introduces the group $G_{n}$ and discusses 
	its role in representing exponential families with functions $C,F_{1},...,F_{n}$.
	Section \ref{nceknkenekn} analyzes the associated orbit space and presents our main result.

\section{Preliminaries}
\subsection{Affine sets and maps.}

	Throughout the paper, $\mathbb{R}^{m}$ is endowed with the Euclidean inner product 
	$\langle x,y\rangle=x_{1}y_{1}+...+x_{m}y_{m}$. We use $\textup{M}(m\times n,\mathbb{R})$, 
	$\textup{M}(m,\mathbb{R})$ and $\textup{GL}(m,\mathbb{R})$ to denote the set of real 
	$m\times n$ matrices, the set of real $m\times m$ 
	matrices and the group of invertible $m\times m$ real matrices, respectively. 

	Let $V$ be a $m$-dimensional real vector space, with basis $\{e_{1},...,e_{m}\}$. 
	If $v\in V$, we write $v=v_{1}e_{1}+...+v_{m}e_{m}$, where $v_{i}\in \mathbb{R}$ for all $i=1,...,m$. 
	We use $\textup{L}(V,W)$ to denote the space of linear maps from $V$ to a vector space $W$, and 
	$\textup{GL}(V)$ to denote the group of invertible linear maps from $V$ to itself.
	
	Let $S$ be a nonempty subset of $V$. The \textit{affine hull} of $S$, denoted by $\textup{aff}(S)$, 
	is the set of all the vectors in $V$ of the form $\lambda_{1}v_{1}+...+\lambda_{s}v_{s}$ such that $v_{i}\in S$ and 
	$\lambda_{1}+...+\lambda_{s}=1$. If $\textup{aff}(S)=S$, we say that $S$ is \textit{affine}. Equivalently, $S\subset V$ 
	is affine if $(1-t)x+ty\in S$ for every $x$ and $y$ in $S$ and $t\in \mathbb{R}$. 
	If $S$ is affine, then the set $\textup{dir}(S):=\{x-y\,\,\,\vert\,\,\,x\in S, y\in S \}$ is a linear subspace of $V$, called 
	the \textit{direction} of $S$. The \textit{dimension} of an affine set is by definition the dimension of its 
	direction. If $S$ is affine, then for every $s\in S$, we have 
	$S=s+\textup{dir}(S):=\{s+u\,\,\vert\,\,u\in \textup{dir}(S)\}$. 

	A map $f$ from $V$ to a vector space $W$ is said to be \textit{affine} if 
	$f(tv_{1}+(1-t)v_{2})=tf(v_{1})+(1-t)f(v_{2})$ for every $v_{1}$ and $v_{2}$ in $V$ and $t\in \mathbb{R}$. 
	Equivalently, a map $f:V\to W$ is affine if there exist $A\in \textup{L}(V,W)$ and $b\in W$ such that 
	$f(v)=A(v)+b$ for all $v\in V$, in which case $A$ and $b$ are unique. 
	We use $\textup{Aff(V,W)}$ to denote the linear space of affine maps from 
	$V$ to $W$, and $\textup{Aff}(V)$ to denote the group invertible affine maps from $V$ to itself. 
	Warning: $\textup{Aff}(V,V)\neq \textup{Aff}(V)$. 

	Let $G$ be a group and $\rho:G\to \textup{GL}(V)$ a group homomorphism. The \textit{semidirect product of $G$ and $V$},
	denoted by $G\ltimes_{\rho} V$, is the Cartesian product $G\times V$ endowed with the product
	\begin{eqnarray*}
		(g,v)\cdot (g',v')=(gg',v+\rho(g)(v')).
	\end{eqnarray*}
	This product turns $G\ltimes_{\rho} V$ into a group. Sometimes we shall leave out the subscript $\rho$ in the notation 
	$G\ltimes_{\rho} V$ if $\rho$ is understood from the context. If $G$ is a Lie group and $\rho$ is a Lie group homomorphism, 
	then $G\ltimes_{\rho} V$ is also a Lie group. For example, the \textit{affine group} $\textup{Aff}(V)$ is 
	by definition the semidirect product $\textup{GL}(V)\ltimes V$, where the group homomorphism 
	is the identity map of $\textup{GL}(V)$.

	If $\rho:G\to \textup{GL}(V)$ is an \textit{anti-homomorphism} of groups, i.e., if 
	$\rho(gg')=\rho(g')\rho(g)$ for all $g$ and $g'$ in $G$, 
	then the semidirect product $G\ltimes_{\rho} V$ is defined similarly, except that the product rule is now 
	\begin{eqnarray}\label{nvrekwnekfnk}
		(g,v)\cdot (g',v')=(gg',v'+\rho(g')(v)).
	\end{eqnarray}

\subsection{Reduction by stages.} In this section, we discuss a simplified version of reduction by stages 
	theory \cite{stages}. The basic idea is that if a ``large group" $G$ acts on a manifold $M$, and if $H$ is a normal subgroup of $G$, 
	then the orbit space $M/G$ can be computed in two stages: first by quotienting by the ``small group" $H$, and then by 
	$G/H$. In other words, $M/G\cong (M/H)/(G/H)$. This procedure often simplifies calculations, 
	but also gives additional information about the fiber structure of the objects involved. 

	We begin by recalling some simple facts of group action theory. Let $G$ be a Lie group that acts smoothly on 
	the left on a manifold $M$. We use $Gx$ to denote 
	the orbit of $x\in M$. The corresponding orbit space is denoted by $M/G$. If we regard orbit spaces as quotient spaces, then 
	we write the equivalence class of $x\in M$ as $[x]_{G}$. 

	The action is said to be \textit{free} if for every $x\in M$, the stabilizer group $G_{x}=\{g\in G\,\,\vert\,\,g\cdot x=x\}$ 
	is trivial. The action is \textit{proper} if the following condition is satisfied: if $(x_{n})_{n\in \mathbb{N}}$ is a 
	convergent sequence in $M$ and if $(g_{n}\cdot x_{n})_{n\in \mathbb{N}}$ converges in $M$, 
	then $(g_{n})_{n\in \mathbb{N}}$ has a convergent subsequence in $G$. If $G$ acts freely and properly on $M$, 
	then the orbit space $M/G$ is naturally a manifold of dimension $\textup{dim}(M)-\textup{dim}(G)$ for which 
	the quotient map $M\to M/G$ is a surjective submersion. 
	
	Let $H$ be a closed Lie subgroup of $G$. Then $H$ acts freely and properly on the left on $G$ by 
	left multiplication, and thus the quotient space $G/H$ is a manifold, called \textit{right coset space}. 
	If in addition $H$ is \textit{normal} 
	in $G$, that is, if $gHg^{-1}=H$ for all $g\in G$, then the formula $(Hg)(Hg')=Hgg'$ defines a 
	multiplication operator on $G/H$ that turns $G/H$ into a Lie group, and the quotient map 
	$G\to G/H$ is a Lie group homomorphims. 

	The next result is a particular case of Poisson reduction by stages (see \cite[Proposition 5.3.1]{stages}). 

\begin{proposition}[\textbf{Reduction by stages}]\label{nfekjkfekkjkfejk}
	Let $G$ be a Lie group that acts freely and properly on the left on a manifold $M$, and let 
	$H$ be a closed Lie subgroup of $G$ that is normal in $G$. 
	\begin{enumerate}[(a)]
	\item $H$ acts freely and properly on $M$ by restriction of the $G$-action.
	\item $G/H$ acts freely and properly on the left on 
		$M/H$ by $[g]_{H}\cdot [x]_{H}=[g\cdot x]_{H}$.
	\item The map $M/G\to (M/H)/(G/H)$, $[x]_{G}\,\mapsto \big[[x]_{H}\big]_{G/H}$ 
		is a diffeomorphism. 
	\end{enumerate}
\end{proposition}
	
	As a typical application of this proposition, consider a Lie group $\widetilde{G}=G\ltimes V$ that is a 
	semidirect product of a Lie group $G$ and a vector space $V$. 
	Then $V\cong \{e\}\times V$ is a normal closed subgroup of $\widetilde{G}$, and the coset space $\widetilde{G}/V$, 
	regarded as a Lie group, is isomorphic to $G$. Thus, if $\widetilde{G}$ acts freely and properly on a manifold $M$, 
	then $V$ and $G$ act freely and properly on $M$ and $M/V$, respectively, and there is a diffeomorphism 
	$M/\widetilde{G}\cong (M/V)/G$.

\subsection{Stiefel manifolds and affine Grassmannians.} \label{nvenwknrfken}
	This section briefly discusses the manifold structure of the affine Grassmannian using a Lie group-theoretical 
	approach involving the Stiefel manifold. For a more systematic treatment, see \cite{lim}.

	Let $V$ be a finite dimensional real vector space. 
	Recall that a \textit{$n$-frame} in $V$ is a linearly independent ordered set of vectors $(v_{1},...,v_{n})$. The set 
	of all $n$-frames in $V$, denoted by $\textup{Stief}_{n}(V)$, is called \textit{Stiefel manifold}. 
	Since $\textup{Stief}_{n}(V)$ is an open subset of the Cartesian product $V^{n}=V\times ...\times V$ ($n$-times), it is 
	naturally a manifold of dimension $n\,\textup{dim}(V)$. 

	The set of all $n$-dimensional affine subsets of $V$, denoted by $\textup{Graff}_{n}(V)$, is 
	called the \textit{affine Grassmannian}. There is a natural projection 
	\begin{eqnarray*}	
	\pi:\,\,\,
		\left\lbrace
		\begin{array}{rcll}
		V\times \textup{Stief}_{n}(V) & \to & \textup{Graff}_{n}(V), \\[0.5em]
		(v_{0},(v_{1},...,v_{n}))	 &\mapsto & v_{0}+\textup{span}\{v_{1},...,v_{n}\}.
		\end{array}
		\right.
	\end{eqnarray*}

	Let $\textup{Aff}(n,\mathbb{R})^{\dag}=\textup{GL}(n,\mathbb{R})\ltimes_{\rho}\mathbb{R}^{n}$ be the semidirect product 
	associated to the \textit{anti-homomorphism} $\rho:\textup{GL}(n,\mathbb{R})\to \textup{GL}(n,\mathbb{R})$, $A\mapsto A^{T}$ (transpose). 
	The product of two elements $(A,x)$ and $(A',x')$ in $\textup{Aff}(\mathbb{R}^{n})^{\dag}$ is given by (see \eqref{nvrekwnekfnk}):
	\begin{eqnarray*}
		(A',x')\cdot (A,x)= (A'A, x+A^{T}x'). 
	\end{eqnarray*}

	As a matter of notation, if $F=(v_{1},...,v_{n})$ is an $n$-frame in $V$, $x=(x_{1},...,x_{n})\in \mathbb{R}^{n}$ and
	$A=(A_{ij})\in \textup{GL}(n,\mathbb{R})$, then we write 
	\begin{eqnarray*}
		\langle x,F\rangle=x_{1}v_{1}+...+x_{n}v_{n}
	\end{eqnarray*}
	and 
	\begin{eqnarray*}
		AF=
		\left(
		\begin{smallmatrix}
		A_{11} & \cdots & A_{1n} \\
		\vdots & \ddots & \vdots \\
		A_{n1} & \cdots & A_{nn}
		\end{smallmatrix}
		\right)
		\left(
		\begin{smallmatrix}
		v_{1}	 \\
		\vdots   \\
		v_{n}
		\end{smallmatrix}
		\right)
		= (A_{11}v_{1}+...+A_{1n}v_{n},...,A_{n1}v_{1}+...+A_{nn}v_{n}). 
	\end{eqnarray*}
	Note that $\langle x,F\rangle\in V$ and $AF\in \textup{Stief}_{n}(V)$. 

			The group $\textup{Aff}(n,\mathbb{R})^{\dag}$ acts on the left on $V\times \textup{Stief}_{n}(V)$ by 
	\begin{eqnarray}\label{nvkenrknknek}
		(A,x)\cdot  (v_{0},F)=(v_{0}+\langle x,F\rangle, AF).
	\end{eqnarray}

	It is immediate that $\pi$ is invariant under the action of $\textup{Aff}(n,\mathbb{R})^{\dag}$ 
	on $V\times \textup{Stief}_{n}(V)$ and hence it descends to a map 
	\begin{eqnarray}\label{nvkwnfekrnkefn}
		\big(V\times \textup{Stief}_{n}(V)\big)/\textup{Aff}(n,\mathbb{R})^{\dag} \to  \textup{Graff}_{n}(V),
	\end{eqnarray}
	which is readily verified to be a bijection. Together with the next lemma, this bijection allows us to endow 
	$\textup{Graff}_{n}(V)$ with a manifold structure.
	
\begin{lemma}
	The action of $\textup{Aff}(n,\mathbb{R})^{\dag}$ on $V\times \textup{Stief}_{n}(V)$ is free and proper.
\end{lemma}
\begin{proof}
	This will be proved in the Section \ref{nceknkenekn}, as an application of the reduction by stages procedure.
\end{proof}

	Consequently, $\big(V\times \textup{Stief}_{n}(V)\big)/\textup{Aff}(n,\mathbb{R})^{\dag}$  is a manifold 
	of dimension $n+n\textup{dim}(V)-(n^{2}+n)=(n+1)(\textup{dim}(V)-n)$. As this quotient space is in bijection with 
	$\textup{Graff}_{n}(V)$, we obtain:
\begin{proposition}\label{nvkednknkefn}
	There exists a unique manifold structure on $\textup{Graff}_{n}(V)$ of dimension $(n+1)(\textup{dim}(V)-n)$ 
	such that \eqref{nvkwnfekrnkefn} is a diffeomorphism. 
\end{proposition}

\begin{remark}
	In this section, we have used the \textit{left} action of $\textup{Aff}(n,\mathbb{R})^{\dag}$ on $V\times \textup{Stief}_{n}(V)$ 
	to endow $\textup{Graff}_{n}(V)$ with a manifold structure. Alternatively, one could use the \textit{right} 
	action of the standard affine group $\textup{Aff}(n,\mathbb{R})$ 
	on $V\times \textup{Stief}_{n}(V)$ given by $(v_{0},F)\cdot (A,x)=(v_{0}+\langle x,F\rangle, A^{T}F)$.
\end{remark}

\section{Representations of exponential families}\label{neknkrnkefnkw}

	In this section, we introduce a Lie group $G_{n}$ and discuss its role in representing exponential families 
	via functions $C, F_{1}, ..., F_{n}$.  For notational convenience, $C(X,Y)$ denotes the set of maps from $X$ to $Y$; 
	we write $C(X)$ when $Y = \mathbb{R}$.  For $f\in C(X, \mathbb{R}^{n})$, we use the component notation 
	$f = (f_{1}, ..., f_{n})$, where $f_{i} \in C(X)$. The space $\mathbb{R}^{n}$ 
	is endowed with the standard Euclidean inner product $\langle u,v\rangle=u_{1}v_{1}+...+u_{n}v_{n}$.  \\

	Let $\Omega=\{x_{0},...,x_{m}\}$ be a finite set. Given $C\in C(\Omega)$ and $F\in C(\Omega,\mathbb{R}^{n})$, define 
	a function $p_{C,F}:\Omega\times \mathbb{R}^{n}\to \mathbb{R}$ by 
	\begin{eqnarray}\label{nvdknkfnknk}
		p_{C,F}(x;\theta)=\textup{exp}\{C(x)+\langle \theta,F(x)\rangle-\psi_{C,F}(\theta)\}, 
	\end{eqnarray}
	where $\psi_{C,F}:\mathbb{R}^{n}\to \mathbb{R}$ is uniquely determined by the condition $\sum_{x\in \Omega}p_{C,F}(x;\theta)=1$ 
	for all $\theta\in \mathbb{R}^{n}$. It is immediate that 
	\begin{eqnarray*}
		\psi_{C,F}(\theta)=\ln\bigg(\sum_{x\in \Omega}\textup{exp}(C(x)+\langle \theta,F(x)\rangle\bigg). 
	\end{eqnarray*}

	For a fixed $\theta\in \mathbb{R}^{n}$, the function 
	$p_{C,F}(.;\theta):\Omega\to \mathbb{R}$ is a probability density function on $\Omega$. 
	The collection of all such probability density functions, indexed by $\theta\in \mathbb{R}^{n}$, forms 
	a set that we will denote by $\mathcal{E}_{C,F}$. Thus 
	\begin{eqnarray*}
		\mathcal{E}_{C,F}=\{p_{C,F}(.;\theta)\in C(\Omega)\,\,\vert\,\,\theta\in \mathbb{R}^{n}\}.
	\end{eqnarray*}

\begin{definition}\label{def:5.3} 
	A set of functions $\mathcal{E}\subset C(\Omega)$ is said to be an 
	\textit{exponential family} of dimension $n$ if there exist $C\in C(\Omega)$ and 
	$F\in C(\Omega,\mathbb{R}^{n})$ such that $\mathcal{E}=\mathcal{E}_{C,F}$. 
\end{definition}

	Since exponential families are defined as sets of functions, two exponential families $\mathcal{E}$ and $\mathcal{E}'$ 
	of dimension $n$ defined on $\Omega$ are equal if and only if they consist of precisely the same functions 
	within the space $C(\Omega)$. 

	An exponential family $\mathcal{E}\subset C(\Omega)$ of dimension $n$ is said to be \textit{mininal} 
	if there are $C\in C(\Omega)$ and $F\in C(\Omega,\mathbb{R}^{n})$ such that 
	$\mathcal{E}=\mathcal{E}_{C,F}$ and the family of functions $\{1,F_1,...,$ $F_n\}$ is 
	linearly independent. In this case, the map $\mathbb{R}^{n}\to \mathcal{E}_{C,F}$, given by 
	$\theta\mapsto p_{C,F}(.;\theta)$, is a bijection, hence defining a global chart for $\mathcal{E}$. 
	The parameters $\theta_{1},...,\theta_{n}$ are called the 
	\textit{natural} or \textit{canonical parameters} of the exponential family.
	The function $\psi_{C,F}$ is known as the \textit{log partition function}, or \textit{cumulant function}. 

	As a matter of notation, we will use 
	\begin{itemize}
	\item $e\textup{-Fam}_{n}(\Omega)$ to denote the set of all 
		minimal exponential families of dimension $n$ defined on $\Omega$, and 
	\item $\textup{Stief}_{n}(C(\Omega))^{\times}$ to denote the set of $n$-frames 
		$F=(F_{1},....F_{n})$ in $\textup{Stief}_{n}(C(\Omega))$ such that the functions 
		$1,F_{1},...,F_{n}$ are linearly independent. 
	\end{itemize}
	Clearly there is a surjective map: 
	\begin{eqnarray*}
		\Phi\,\,:\,\,C(\Omega)\times \textup{Stief}_{n}(C(\Omega))^{\times} \,\,\,\to \,\,\,e\textup{-Fam}_{n}(\Omega),
	\end{eqnarray*}
	defined by $\Phi(C,F)=\mathcal{E}_{C,F}$. 
	Two elements $(C,F)$ and $(C',F')$ in $C(\Omega)\times \textup{Stief}_{n}(C(\Omega))^{\times}$
	are said to be \textit{equivalent} if their images under $\Phi$ coincide, that is, 
	if $\mathcal{E}_{C,F}=\mathcal{E}_{C',F'}$ (equality of subsets of $C(\Omega)$). 
	This defines an equivalence relation $\sim$ on the set $C(\Omega)\times \textup{Stief}_{n}(C(\Omega))^{\times}$, 
	and $\Phi$ descends to a bijection:

	\begin{eqnarray*}
		\big(C(\Omega)\times \textup{Stief}_{n}(C(\Omega))^{\times}\big)/{\,\sim} \,\,\,\to\,\,\,  e\textup{-Fam}_{n}(\Omega). 
	\end{eqnarray*}

	The quotient space will be investigated in the next section, where we will perform a reduction by stages. 
	For now, we focus on the properties of $\sim$. To this end, 
	let $G_{n}\subset \textup{GL}(n+2,\mathbb{R})$ be the group of matrices of the form 

	\begin{eqnarray*}
	\begin{bmatrix}
	1  &  u   &   c  \\
	0  &  A   &   v  \\
	0  &  0   &   1
	\end{bmatrix}
	=
	\begin{bmatrix}
	1       & u_{1}   & \cdots & u_{n}   & c      \\
	0       & A_{11}  & \cdots & A_{1n}  & v_{1}  \\
	\vdots  &  \vdots & \ddots & \vdots  & \vdots \\
	0       & A_{n1}  & \cdots & A_{nn}  & v_{n}  \\
	0       & \cdots  &    0   &   0     &   1
	\end{bmatrix},
	\end{eqnarray*}
	where $A=(A_{ij})\in \textup{GL}(n,\mathbb{R})$, $u=(u_{1},...,u_{n})$ and $v=(v_{1},...,v_{n})$ 
	are vectors in $\mathbb{R}^{n}$ and $c\in \mathbb{R}$.
	The product of two elements in $G_{n}$ is given by 

	\begin{eqnarray}\label{nceknkefnkn}
	\begin{bmatrix}
	1  &  u'   &   c'  \\
	0  &  A'   &   v'  \\
	0  &  0   &   1
	\end{bmatrix}
	\begin{bmatrix}
	1  &  u   &   c  \\
	0  &  A   &  v  \\
	0  &  0   &   1
	\end{bmatrix}
	=\begin{bmatrix}
	1  &  u+A^{T}u'   &   c'+c+\langle u',v\rangle  \\
	0  &  A'A            &   v'+A'v                     \\
	0  &  0              &   1
	\end{bmatrix},
	\end{eqnarray}
	where $A^{T}$ denotes the transpose of $A$. The inverse of an element in $G_{n}$ is given by 

	\begin{eqnarray*}
	\begin{bmatrix}
	1  &  u   &   c  \\
	0  &  A   &   v  \\
	0  &  0   &   1
	\end{bmatrix}^{-1}
	= \begin{bmatrix}
	1  &  -(A^{-1})^{T}u   &   -c+\langle u,A^{-1}v\rangle  \\
	0  &  A^{-1}           &   -A^{-1}v                     \\
	0  &  0   &   1
	\end{bmatrix}.
	\end{eqnarray*}
	The group $G_{n}$ acts on the left on $C(\Omega)\times \textup{Stief}_{n}(C(\Omega))^{\times}$ by 

	\begin{eqnarray*}
	\begin{bmatrix}
	1  &  u   &   c  \\
	0  &  A   &   v  \\
	0  &  0   &   1
	\end{bmatrix}
	\cdot 
	(C,F)= (C+\langle u,F\rangle+c, AF+v)
	\end{eqnarray*}	
	(see Section \ref{nvenwknrfken} for the notation). In the formula above, 
	$c$ and $v$ are interpreted as constant functions in $C(\Omega)$ and 
	$C(\Omega,\mathbb{R}^{n})$, respectively.
	
\begin{proposition}\label{neknkrnkefnk}
	Let $\Omega=\{x_{0},...,x_{m}\}$ be a finite set. Two elements $(C,F)$ and $(C',F')$ in 
	$C(\Omega)\times \textup{Stief}_{n}(C(\Omega))^{\times}$ are equivalent if and only if there is 
	$g=\begin{bsmallmatrix} 
	1   &   u    &   c  \\
	0   &   A    &   v  \\
	0   &   0    &   1  \\
	\end{bsmallmatrix}\in G_{n}$ 
	such that $(C,F)=g\cdot (C',F')$. In this case, the following hold. 
	\begin{enumerate}[(a)]
	\item $p_{C,F}(x;\theta)=p_{C',F'}(x;A^{T}\theta+u)$ for all $x\in \Omega$ and all $\theta\in \mathbb{R}^{n}$. 
	\item $\psi_{C,F}(\theta)=\psi_{C',F'}(A^{T}\theta+u)+\langle \theta,v\rangle+c$ for all $\theta\in \mathbb{R}^{n}$. 
	\end{enumerate}
\end{proposition}
\begin{proof}
	Suppose that $(C,F)$ and $(C',F')$ are equivalent. By minimality of the exponential families $\mathcal{E}_{C,F}$ and $\mathcal{E}_{C',F'}$, 
	there is a unique bijection 
	$\varphi:\mathbb{R}^{n}\to \mathbb{R}^{n}$ such that $p_{C,F}(x,\theta)=p_{C',F'}(x,\varphi(\theta))$ for all 
	$\theta\in \mathbb{R}$, and so 
	\begin{equation}
		\tag{$E_{i}$}
		C(x_{i})+\langle \theta,F(x_{i})\rangle-\psi(\theta)=C'(x_{i})+\langle \varphi(\theta),F'(x_{i})\rangle-\psi'(\varphi(\theta))
		\label{vrknfknkrnfk}
	\end{equation}
	for all $\theta\in \mathbb{R}^{n}$ and all $i\in \{0,1,...,m\}$ (see \eqref{nvdknkfnknk}). 
	Subtracting $(E_{j})$ from $(E_{i})$, $i,j\in \{0,...,m\}$, we get 
	\begin{eqnarray}\label{nfeknkenkwnk}
		\lefteqn{C(x_{i})-C(x_{j})+\langle \theta,F(x_{i})-F(x_{j})\rangle }\nonumber\\
		&=&C'(x_{i})-C'(x_{j})+\langle \varphi(\theta),F'(x_{i})-F'(x_{j})\rangle. 
	\end{eqnarray}
	Because $F_{1}',...,F_{n}'$ and $1$ are linearly independent, there 
	are indices $0\leq i_{0}<...<i_{n}\leq m$ such that the matrix 
	\begin{eqnarray*}
		M'=
		\begin{bsmallmatrix}
			F'_{1}(x_{i_{1}})- F'_{1}(x_{i_{0}}) & \cdots &   F'_{1}(x_{i_{n}}) - F'_{1}(x_{i_{0}})   \\
			    \vdots   &        &     \vdots       \\ 
			F'_{n}(x_{i_{1}})- F'_{n}(x_{i_{0}}) & \cdots &   F'_{n}(x_{i_{n}}) - F'_{n}(x_{i_{0}}) 
		\end{bsmallmatrix}
	\end{eqnarray*}
	is invertible (see Lemma \ref{nekwnknfekwnkn} below). Let $N'$ be the $n\times 1$ matrix defined by 
	\begin{eqnarray*}
	N'=\begin{bsmallmatrix} 
		C'(x_{i_{1}})-C'(x_{i_{0}}) \\
		\vdots    \\
		C'(x_{i_{n}})-C'(x_{i_{0}})
	\end{bsmallmatrix}. 
	\end{eqnarray*}
	Similarly, we define the matrices $M$ and $N$ (with $F$ and $C$ in place of $F'$ and $C'$, respectively). 
	Let $A^{T}$ denote the transposee of a matrix $A$. It follows from \eqref{nfeknkenkwnk} that 
	\begin{eqnarray*}
		N+M^{T}\theta=N'+(M')^{T}\varphi(\theta), 
	\end{eqnarray*}
	where $\theta$ and $\varphi(\theta)$ are regarded as column vectors. Since $M'$ is invertible, we get 
	\begin{eqnarray*}
		\varphi(\theta)=((M')^{T})^{-1}M^{T}\theta+((M')^{T})^{-1}(N-N')=L\theta+u
	\end{eqnarray*}
	for all $\theta\in \mathbb{R}^{n}$, where $L=((M')^{T})^{-1}M^{T}$ and $u=((M')^{T})^{-1}(N-N')$. 
	Therefore $\varphi$ is affine. Note that $L$ is invertible, since $\varphi$ is bijective.

	Taking the derivative of \eqref{vrknfknkrnfk} with respect to $\theta_{k}$ and 
	using $\varphi(\theta)=L\theta+u$ we see that
	\begin{eqnarray*}
		F_{k}(x)-(L^{T}F'(x))_{k}=\dfrac{\partial}{\partial \theta_{k}}\big[\psi(\theta)-\psi'(L\theta+u)\big]
	\end{eqnarray*}
	for all $x\in \Omega$ and all $\theta\in \mathbb{R}^{n}$. Left and right sides depend on different variables and hence 
	there is a vector $v=(v_{1},...,v_{n})\in \mathbb{R}^{n}$ such that 
	\begin{eqnarray}\label{nenwkdknfknd}
		F(x)-L^{T}F'(x)=v
	\end{eqnarray}
 	for all $x\in \Omega$ and 
	\begin{eqnarray}\label{nfenkwknkj}
		\tfrac{\partial}{\partial \theta_{k}}\big[\psi(\theta)-\psi'(L\theta+u)\big]=v_{k}
	\end{eqnarray}
	for all $k\in \{1,...,n\}$ and all $\theta\in \mathbb{R}^{n}$.
	It follows from \eqref{nfenkwknkj} that there exists $c\in \mathbb{R}$ such that 
	\begin{eqnarray}\label{nfrfenwkernfekn}
		\psi(\theta)=\psi'(L\theta+u)+\langle \theta,v\rangle+c
	\end{eqnarray}
	for all $\theta\in \mathbb{R}^{n}$. Then, using the formula $p_{C,F}(x,\theta)=p_{C',F'}(x,L\theta+u)$, \eqref{nenwkdknfknd} 
	and \eqref{nfrfenwkernfekn} we see that 
	\begin{eqnarray}\label{nefknkneknfkn}
		C(x)=C'(x)+\langle u,F'(x)\rangle+c
	\end{eqnarray}
	for all $x\in \Omega$. It follows from \eqref{nenwkdknfknd} and \eqref{nefknkneknfkn} that 
	$(C,F)=
	\begin{bsmallmatrix} 
	1   &   u    &   c  \\
	0   &   A    &   v  \\
	0   &   0    &   1  \\
	\end{bsmallmatrix}\cdot (C',F')$, where $A=L^{T}$. The discussion above also shows (a) and (b). 
	The converse is readily verified, completing the proof.
\end{proof}

\begin{lemma}\label{nekwnknfekwnkn}
	Let $F_{1},...,F_{n}$ be real-valued functions defined on the finite set $\Omega=\{x_{0},...,x_{m}\}$. 
	If $F_{1},...,F_{n}$ and the constant function $1$ are linearly independent, then there are indices 
	$0\leq i_{0}<...<i_{n}\leq m$ such that the following matrix is invertible:
	\begin{eqnarray*}
	\begin{bmatrix}
			F_{1}(x_{i_{1}})- F_{1}(x_{i_{0}}) & \cdots &   F_{1}(x_{i_{n}})- F_{1}(x_{i_{0}})   \\
			    \vdots   &        &     \vdots       \\ 
			F_{n}(x_{i_{1}})- F_{n}(x_{i_{0}}) & \cdots &   F_{n}(x_{i_{n}}) - F_{n}(x_{i_{0}}) 
		\end{bmatrix}.
	\end{eqnarray*}
\end{lemma}
\begin{proof}
	By hypothesis, the rank of the matrix
	\begin{eqnarray*}
		\begin{bsmallmatrix}
			1            & \cdots &   1              \\
			F_{1}(x_{0}) & \cdots &   F_{1}(x_{m})   \\
			    \vdots   &        &     \vdots       \\ 
			F_{n}(x_{0}) & \cdots &   F_{n}(x_{m})
		\end{bsmallmatrix}
	\end{eqnarray*}
	is $n+1$, and thus there are indices $0\leq i_{0}<...<i_{n}\leq m$ such that 
	\begin{eqnarray*}
		\begin{vsmallmatrix}
			1            & \cdots &   1              \\
			F_{1}(x_{i_{0}}) & \cdots &   F_{1}(x_{i_{n}})   \\
			    \vdots   &        &     \vdots       \\ 
			F_{n}(x_{i_{0}}) & \cdots &   F_{n}(x_{i_{n}})
		\end{vsmallmatrix}
		=\begin{vsmallmatrix}
			F_{1}(x_{i_{1}})- F_{1}(x_{i_{0}}) & \cdots &   F_{1}(x_{i_{n}})- F_{1}(x_{i_{0}})   \\
			    \vdots   &        &     \vdots       \\ 
			F_{n}(x_{i_{1}})- F_{n}(x_{i_{0}}) & \cdots &   F_{n}(x_{i_{n}}) - F_{n}(x_{i_{0}}) 
		\end{vsmallmatrix}
		\neq 0. 
		\qedhere
	\end{eqnarray*}
\end{proof}

\begin{remark}\label{nceknwknekefnkn}
	The lemma implies that $\textup{Stief}_{n}(C(\Omega))^{\times}$ is a finite union of open subsets 
	of $C(\Omega)\times ... \times C(\Omega)$ ($n$-times), and thus is open. 
\end{remark}

\section{Exponential families and Grassmannians}\label{nceknkenekn}


	In this section, we show that the orbit space $\big(C(\Omega)\times \textup{Stief}_{n}(C(\Omega))^{\times}\big)/G_{n}$ 
	is an affine Grassmannian. We achieve this by first noting that $G_{n}$ has the structure of a 
	semidirect product, which allows us to perform a reduction by stages.

	Let $\textup{Aff}(\mathbb{R}^{n})^{\dag}=\textup{GL}(\mathbb{R}^{n}) \ltimes_{\rho} \mathbb{R}^{n}$ be the 
	semidirect product associated to the anti-homomorphism 
	$\rho:\textup{GL}(n,\mathbb{R})\to \textup{GL}(n,\mathbb{R})$, $A\mapsto A^{T}$ (transpose). 
	Let $\varepsilon:\textup{Aff}(\mathbb{R}^{n})^{\dag} \to \textup{GL}(\mathbb{R}^{n}\oplus \mathbb{R})$ be the homomorphism defined by 
	\begin{eqnarray*}
		\varepsilon(A,u)(v,c):=(Av,\langle u,v\rangle+c). 
	\end{eqnarray*}
	Denote by $\textup{Aff}(\mathbb{R}^{n})^{\dag} \ltimes (\mathbb{R}^{n}\oplus 
	\mathbb{R})$ the semidirect product associated to $\varepsilon$ with multiplication 
	\begin{eqnarray}\label{nvfeekwnkefnkn}
		\lefteqn{\big((A',u'),(v',c')\big)\big((A,u),(v,c)\big)}\nonumber \\
		&=&\big((A'A,u+A^{T}u'),(v'+A'v,c'+c+\langle u',v\rangle)\big). 
	\end{eqnarray}

\begin{lemma}
	The map
	\begin{eqnarray*}
		\left\lbrace
		\begin{array}{lll}
			\textup{Aff}(\mathbb{R}^{n})^{\dag}\ltimes_{} 
				(\mathbb{R}^{n}\oplus\mathbb{R})	&\to& G_{n}, \\[0.5em]
			\big((A,u),(v,c)\big)	&\mapsto & \Big[\begin{smallmatrix}
								1 & u  & c \\
								0 & A  & v \\
								0 & 0  & 1
							   \end{smallmatrix}\Big],
		\end{array}
		\right.
	\end{eqnarray*}
	is a Lie group isomorphism.
\end{lemma}
\begin{proof}
	By a direct verification using \eqref{nceknkefnkn} and \eqref{nvfeekwnkefnkn}.
\end{proof}

	In what follows, we will identify $G_{n}$ and $\textup{Aff}(\mathbb{R}^{n})^{\dag}\ltimes_{}(\mathbb{R}^{n}\oplus\mathbb{R})$ 
	whenever convenient. In particular, we will regard $\mathbb{R}^{n}\oplus \mathbb{R}$ as a normal closed subgroup of $G_{n}$. \\

	Let $\Omega=\{x_{0},...,x_{m}\}$ be a finite set. Recall that $G_{n}$ acts on $C(\Omega)\times \textup{Stief}_{n}(C(\Omega))^{\times}$ by 
	\begin{eqnarray}\label{nvkewnkenkn}
		\big((A,u),(v,c)\big)\cdot (C,F)=(C+\langle u,F\rangle+c,AF+v)
	\end{eqnarray}
	(see Section \ref{nvenwknrfken} for the notation). This action is smooth when 
	$C(\Omega)\times \textup{Stief}_{n}(C(\Omega))^{\times}$ is considered as an open subset of the 
	$(n+1)$-fold Cartesian product $C(\Omega)\times ...\times C(\Omega)$ (see Remark 2.4).

	We wish to reduce 
	$C(\Omega)\times \textup{Stief}_{n}(C(\Omega))^{\times}$ by the action of $G_{n}$ in two stages, 
	first by the normal subgroup $\mathbb{R}^{n}\oplus\mathbb{R}$, and then by 
	$G_{n}/(\mathbb{R}^{n}\oplus \mathbb{R})=\textup{Aff}(\mathbb{R}^{n})^{\dag}$. 
	We begin by showing that the $G_{n}$-action satisfies the reduction by stages hypothesis (see Proposition \ref{nfekjkfekkjkfejk}): 

\begin{lemma}
	The action of $G_{n}$ on $C(\Omega)\times \textup{Stief}_{n}(C(\Omega))^{\times}$ is free and proper. 
\end{lemma}
\begin{proof}
	First we prove that the action is free. Suppose that $\big((A,u),(v,c)\big)\in G_{n}$ is in the stabilizer of $(C,F)$. 
	Thus
	\begin{eqnarray}\label{nfekwnekefnkn}
		\left\lbrace
		\begin{array}{lll}
			\langle u,F\rangle +c &=& 0, \\
			AF+v                  &=& F. 
		\end{array}
		\right.
	\end{eqnarray}
	Since $1,F_{1},...,F_{n}$ are linearly independent, the first equation implies that 
	$u=0$ and $c=0.$ To see that $A$ is the identity matrix, we apply Lemma \ref{nekwnknfekwnkn} to $F$: there are indices 
	$0\leq i_{0}<i_{1}<...<i_{n}\leq m$ such that the $n\times n$ matrix $\Omega$, whose $(a,b)$-entry is 
	$\Omega_{ab}= F_{a}(x_{i_{b}})-F_{a}(x_{i_{0}})$, is invertible. The matrix $\Omega$ is related to $A$ 
	by $A\Omega=\Omega$ (this follows from the second equation of \eqref{nfekwnekefnkn}), and 
	thus $A=I_{n}$. The fact that $v=0$ is immediate from the second equation of \eqref{nfekwnekefnkn}. This 
	completes the proof that the $G_{n}$-action is free. 

	Next we show that the $G_{n}$-action is proper. Let 
	$\big((A_{\alpha},u_{\alpha})(v_{\alpha},c_{\alpha})\big)_{\alpha\in \mathbb{N}}$
	and $(C_{\alpha},F_{\alpha})_{\alpha\in \mathbb{N}}$ be sequences of points in $G_{n}$ and 
	$C(\Omega)\times \textup{Stief}_{n}(C(\Omega))^{\times}$, respectively, such that 
	\begin{itemize}	
		\item $(C_{\alpha},F_{\alpha})\to (C^{*},F^{*})\in C(\Omega)\times \textup{Stief}_{n}(C(\Omega))^{\times}$ 
			and 
		\item $\Big[\begin{smallmatrix}
			1 & u_{\alpha}  & c_{\alpha} \\
			0 & A_{\alpha}  & v_{\alpha} \\
			0 & 0  & 1
		   \end{smallmatrix}\Big]\cdot (C_{\alpha},F_{\alpha})=:
			(G_{\alpha},H_{\alpha})\to (G^{*},H^{*})\in C(\Omega)\times \textup{Stief}_{n}(C(\Omega))^{\times}$ 
	\end{itemize}
	as $\alpha\to \infty$. We must show that $\big((A_{\alpha},u_{\alpha})(v_{\alpha},c_{\alpha})\big)_{\alpha\in \mathbb{N}}$ 
	has a convergent subsequence in $G_{n}$. The second condition above implies that 
	\begin{eqnarray}\label{nveeknwkenfk}
	\left\lbrace
	\begin{array}{lll}
		C_{\alpha}+\langle u_{\alpha},F_{\alpha}\rangle+c_{\alpha} &=& G_{\alpha}, \\[0.6em]
		A_{\alpha}F_{\alpha}+v_{\alpha}	                           &=& H_{\alpha},
	\end{array}
	\right.
	\end{eqnarray}
	for all $\alpha\in \mathbb{N}$. We will show first that $(A_{\alpha})_{\alpha\in \mathbb{N}}$ converges 
	to an invertible matrix. To see this, we apply 
	Lemma \ref{nekwnknfekwnkn} to $F^{*}$: there are indices 
	$0\leq j_{0}<j_{1}<...<j_{n}\leq m$ such that the $n\times n$ matrix $\Omega^{*}$, whose $(a,b)$-entry is 
	$\Omega^{*}_{ab}= F^{*}_{a}(x_{j_{b}})-F^{*}_{a}(x_{j_{0}})$, is invertible. Let $\Omega_{\alpha}$, 
	$\Psi_{\alpha}$ and $\Psi^{*}$ be the $n\times n$ matrices whose $(a,b)$-entries are 
	\begin{eqnarray*}
	(\Omega_{\alpha})_{ab} &=& (F_{\alpha})_{a}(x_{j_{b}})-(F_{\alpha})_{a}(x_{j_{0}}),\\
	(\Psi_{\alpha})_{ab}   &=& (H_{\alpha})_{a}(x_{j_{b}})-(H_{\alpha})_{a}(x_{j_{0}}),\\
	(\Psi^{*})_{ab}   &=&  (H^{*})_{a}(x_{j_{b}})-(H^{*})_{a}(x_{j_{0}}),
	\end{eqnarray*}
	respectively. Note that:
	\begin{enumerate}[(1)]
	\item    $\Omega_{\alpha}\to \Omega^{*}$ and $\Psi_{\alpha}\to \Psi^{*}$ as $\alpha\to \infty$ 
		 (convergence in $\textup{M}(n,\mathbb{R})$).
	\item   $\Omega_{\alpha}$ is invertible when $\alpha$ is large enough. 
	\item $A_{\alpha}\Omega_{\alpha}=\Psi_{\alpha}$ for all $\alpha\in \mathbb{N}$. 
	\end{enumerate}
	The first assertion follows from the fact that $(F_{\alpha})_{\alpha\in \mathbb{N}}$ and 
	$(H_{\alpha})_{\alpha\in \mathbb{N}}$ converge 
	to $F^{*}$ and $H^{*}$, respectively. The second assertion is a consequence of (1) and the 
	fact that $\Omega^{*}$ is invertible. The third assertion follows from 
	the second equation of \eqref{nveeknwkenfk}.
	
	It follows from (2) and (3) that $A_{\alpha}=\Psi_{\alpha}(\Omega_{\alpha})^{-1}$ 
	for all $\alpha\in \mathbb{N}$ large enough, and thus 
	$A_{\alpha}$ converges to $A^{*}:=\Psi^{*}(\Omega^{*})^{-1}$ in $\textup{M}(n,\mathbb{R})$. To see that $A^{*}$ is invertible, 
	we apply again Lemma \ref{nekwnknfekwnkn}, but this time to $H^{*}$, and carry out a discussion analogous to that above.
 	We find that there are two $n\times n$ matrices, say $P$ and $Q$, with $Q$ invertible, 
	such that $A^{*}P=Q$, which forces $A^{*}$ to be invertible.

	Finally, we prove that the sequences $(u_{\alpha})_{\alpha\in \mathbb{N}}$, $(v_{\alpha})_{\alpha\in \mathbb{N}}$ and 
	$(c_{\alpha})_{\alpha\in \mathbb{N}}$ converge. Clearly, the second equation of \eqref{nveeknwkenfk} implies that 
	$(v_{\alpha})_{\alpha\in \mathbb{N}}$ converges. To see that $(u_{\alpha})_{\alpha\in \mathbb{N}}$ converges, 
	we will use the matrices $\Omega^{*}$ and $\Omega_{\alpha}$ defined above. In view of the first equation of 
	\eqref{nveeknwkenfk}, we have $\Omega_{\alpha}^{T}u_{\alpha}=R_{\alpha}$ for all $\alpha\in \mathbb{N}$, where $\Omega_{\alpha}^{T}$ 
	is the transpose of $\Omega_{\alpha}$, and $R_{\alpha}$ is the vector in $\mathbb{R}^{n}$ defined by 
	\begin{eqnarray*}
		R_{\alpha}=\begin{bmatrix}
		G_{\alpha}(x_{j_{1}})-G_{\alpha}(x_{j_{0}})-(C_{\alpha}(x_{j_{1}})-C_{\alpha}(x_{j_{0}}))\\
		\vdots\\
		G_{\alpha}(x_{j_{n}})-G_{\alpha}(x_{j_{0}})-(C_{\alpha}(x_{j_{n}})-C_{\alpha}(x_{j_{0}}))
		\end{bmatrix}.
	\end{eqnarray*}
	Since $(C_{\alpha})_{\alpha\in \mathbb{N}}$ and $(G_{\alpha})_{\alpha\in \mathbb{N}}$ converge, 
	$R_{\alpha}$ converges to some $R^{*}\in \mathbb{R}^{n}$, and because $\Omega_{\alpha}$ converges 
	to the invertible matrix $\Omega^{*}$, $(u_{\alpha})_{\alpha\in \mathbb{N}}$ converges to 
	$((\Omega^{*})^{-1})^{T}R^{*}$. The fact that $(c_{\alpha})_{\alpha\in \mathbb{N}}$ converges is then obvious. 
\end{proof}

	We now proceed with the reduction by stages. The induced action of the group $\mathbb{R}^{n}\oplus \mathbb{R}\subset G_{n}$ on 
	$C(\Omega)\times \textup{Stief}_{n}(C(\Omega))^{\times}$ is given by 
	\begin{eqnarray}\label{nfdksdnkfnkn}
		(v,c)\cdot (C,F)=(C+c,F+v).
	\end{eqnarray}
	We will use $[(C,F)]$ to denote the equivalence class of $(C,F)$ with respect to the action of 
	$\mathbb{R}^{n}\oplus \mathbb{R}$. If $C\in C(\Omega)$, we will use $[C]$ to denote the equivalence class of $C$ 
	in $C(\Omega)/\mathbb{R}$ (here $\mathbb{R}$ is identified with the space of constant functions on $\Omega$).

\begin{lemma}[\textbf{First reduction}]
	The map 
	\begin{eqnarray}\label{nceejnefnknk}
		\left\lbrace
		\begin{array}{rcl}
			\big(C(\Omega)\times \textup{Stief}_{n}(C(\Omega))^{\times}\big)/(\mathbb{R}^{n}\oplus \mathbb{R})
			&\to  &
			C(\Omega)/\mathbb{R}\times \textup{Stief}_{n}\big(C(\Omega)/\mathbb{R}\big),\\[0.5em]
			\textbf{}[(C,(F_{1},...,F_{n})]  &\mapsto & \big([C], ([F_{1}],...,[F_{n}])\big). 
		\end{array}
		\right.
	\end{eqnarray}
	is a diffeomorphism.
\end{lemma}
\begin{proof}
	This follows easily from standard geometric-differential arguments (in particular, \cite[Theorem 4.29]{LeeSmooth}), and the following 
	simple facts: 
	\begin{enumerate}[(1)]
	\item If $F_{1},...,F_{n}\in C(\Omega)$, then $(F_{1},...,F_{n})\in \textup{Stief}_{n}(C(\Omega))^{\times}$ if and only if 
		$([F_{1}],...,[F_{n}])\in \textup{Stief}_{n}(C(\Omega)/\mathbb{R})$. 
	\item The map \begin{eqnarray*}
		g:  \left\lbrace
		\begin{array}{rcl}
		C(\Omega)/\mathbb{R}\times \textup{Stief}_{n}\big(C(\Omega)/\mathbb{R}\big)&\to  &
		\big(C(\Omega)\times \textup{Stief}_{n}(C(\Omega))^{\times}\big)/(\mathbb{R}^{n}\oplus \mathbb{R}),\\[0.5em]
		\big([G], ([H_{1}],...,[H_{n}])\big)  &\mapsto & \textbf{}\big[\big(G,(H_{1},...,H_{n})\big)\big] 
		\end{array}
		\right.
	\end{eqnarray*}
	is an inverse for $f$. \qedhere
	\end{enumerate}
\end{proof}

	In what follows, we will identify 
	$\big(C(\Omega)\times \textup{Stief}_{n}(C(\Omega))^{\times}\big)/(\mathbb{R}^{n}\oplus \mathbb{R})$ and 
	$C(\Omega)/\mathbb{R}\times \textup{Stief}_{n}(C(\Omega)/\mathbb{R})$.

	Now we focus our attention on the second reduction. Given $F=(F_{1},...,F_{n})\in C(\Omega,\mathbb{R}^{n})$, we 
	wite $[F]=([F_{1}],...,[F_{n}])\in C(\Omega)/\mathbb{R}\times ... \times C(\Omega)/\mathbb{R}$. 
	In view of Proposition \ref{nfekjkfekkjkfejk}, the quotient group $G_{n}/(\mathbb{R}^{n}\oplus \mathbb{R})=
	\textup{Aff}(\mathbb{R}^{n})^{\dag}$ acts on the left on $C(\Omega)/\mathbb{R}\times \textup{Stief}_{n}(C(\Omega)/\mathbb{R})$ by 
	\begin{eqnarray}
		(A,u)\cdot ([C], [F]) = (C+\langle u,[F]\rangle, A[F]), 
	\end{eqnarray}
	where $\langle u,[F]\rangle= u_{1}[F_{1}]+...+u_{n}[F_{n}]$ and $A[F]=
	(a_{11}[F_{1}]+...+a_{1n}[F_{n}],...,a_{n1}[F_{1}]+...+a_{nn}[F_{n}])$.

\begin{lemma}[\textbf{Second reduction}]
	The map 
	\begin{eqnarray}\label{nceejnefnknkk}
		\left\lbrace
		\begin{array}{rcl}
			\big(C(\Omega)/\mathbb{R}\times 
			\textup{Stief}_{n}\big(C(\Omega)/\mathbb{R}\big)\big)/\textup{Aff}(\mathbb{R}^{n})^{\dag}
			&\to  & \textup{Graff}_{n}(C(\Omega)/\mathbb{R}),\\[0.5em]
			\textbf{}\big[\big([C],([F_{1}],...,[F_{n}])\big)\big]  &\mapsto & [C]+\textup{span}\{[F_{1}],...,[F_{n}]\},
		\end{array}
		\right.
	\end{eqnarray}
	is a diffeomorphism, where $\big[\big([C],([F_{1}],...,[F_{n}])\big)\big]$ denotes the equivalence class of 
	$([C],([F_{1}],...,[F_{n}]))$ with respect to the action of $\textup{Aff}(\mathbb{R}^{n})^{\dag}$. 
\end{lemma}
\begin{proof}
	This is a straightforward application of Proposition \ref{nvkednknkefn} with $V=C(\Omega)/\mathbb{R}$. 
\end{proof}
	Combining the two reductions, we get: 
\begin{theorem}\label{nekwnkefnkn}
	Let $\Omega=\{x_{0},...,x_{m}\}$ be a finite set. Then the $G_{n}$-action on $C(\Omega)\times \textup{Stief}_{n}(C(\Omega))^{\times}$
	is free and proper (in particular, the corresponding orbit space is a manifold), and the map 
	\begin{eqnarray*}
		\left\lbrace
		\begin{array}{rcl}
			\big(C(\Omega)\times \textup{Stief}_{n}(C(\Omega))^{\times}\big)/G_{n} 
				&\to  & \textup{Graff}_{n}(C(\Omega)/\mathbb{R}),\\[0.5em]
			[(C,F)] &\mapsto & [C]+\textup{span}\{[F_{1}],...,[F_{n}]\},
		\end{array}
		\right.
	\end{eqnarray*}
	is a diffeomorphism. 
\end{theorem}

	Recall that $e\textup{-Fam}_{n}(\Omega)$ denotes the set of $n$-dimensional minimal exponential families 
	defined on the set $\Omega$. The next result is an immediate consequence of Theorem \ref{nekwnkefnkn} and 
	our discussion in Section \ref{neknkrnkefnkw}. 

\begin{corollary}\label{nfeknkernkfn}
	The map 
	\begin{eqnarray*}
		\left\lbrace
		\begin{array}{rcl}
		e\textup{-Fam}_{n}(\Omega) &\to     & \textup{Graff}_{n}(C(\Omega)/\mathbb{R}), \\[0.5em]
		\mathcal{E}_{C,F}          &\mapsto &  [C]+\textup{span}\{[F_{1}],...,[F_{n}]\},
		\end{array}
		\right.
	\end{eqnarray*}
	is a bijection. 
\end{corollary}

\begin{example}
	The corollary implies that there exists only one minimal exponential family of dimension 
	$n$ defined on a sample space $\Omega$ of cardinality $n+1$. 
\end{example}

\section*{Acknowledgement}
	The first author was supported by the \textit{Coordena\c{c}\~ao de Aperfeiçoamento de Pessoal de N\'ivel Superior -- 
	Brasil (CAPES) -- Finance Code 001}. 

\begin{footnotesize}\bibliography{bibtex}\end{footnotesize}

\newcommand{\etalchar}[1]{$^{#1}$}
\begin{thebibliography}{MMeO{\etalchar{+}}07}

\bibitem[BN14]{nielsen}
O.~Barndorff-Nielsen.
\newblock {\em Information and exponential families in statistical theory}.
\newblock Wiley Series in Probability and Statistics. John Wiley \& Sons, Ltd.,
  Chichester, 2014.
\newblock Reprint of the 1978 original [MR0489333].

\bibitem[Bro86]{brown}
Lawrence~D. Brown.
\newblock {\em Fundamentals of statistical exponential families with
  applications in statistical decision theory}, volume~9 of {\em Institute of
  Mathematical Statistics Lecture Notes---Monograph Series}.
\newblock Institute of Mathematical Statistics, Hayward, CA, 1986.

\bibitem[LC98]{casella}
E.~L. Lehmann and George Casella.
\newblock {\em Theory of point estimation}.
\newblock Springer Texts in Statistics. Springer-Verlag, New York, second
  edition, 1998.

\bibitem[Lee13]{LeeSmooth}
John~M. Lee.
\newblock {\em Introduction to smooth manifolds}, volume 218 of {\em Graduate
  Texts in Mathematics}.
\newblock Springer, New York, second edition, 2013.

\bibitem[LWY21]{lim}
Lek-Heng Lim, Ken Sze-Wai Wong, and Ke~Ye.
\newblock The {G}rassmannian of affine subspaces.
\newblock {\em Found. Comput. Math.}, 21(2):537--574, 2021.

\bibitem[MMeO{\etalchar{+}}07]{stages}
Jerrold~E. Marsden, Gerard Misio\l~ek, Juan-Pablo Ortega, Matthew Perlmutter,
  and Tudor~S. Ratiu.
\newblock {\em Hamiltonian reduction by stages}, volume 1913 of {\em Lecture
  Notes in Mathematics}.
\newblock Springer, Berlin, 2007.

\end{thebibliography}
\end{document}